\documentclass[12pt,leqno]{article}
\usepackage[latin1]{inputenc} %scrive in latin e si puo' scrivere direttamente accentato 
\usepackage[T1]{fontenc} 
\usepackage[english]{babel} 
\usepackage{indentfirst} %fa iniziare i paragrafi allineati 
\usepackage{latexsym}
\usepackage{amssymb}
\usepackage{amsmath}
\usepackage{amsthm}
\usepackage{url}
\usepackage{enumitem}	
\usepackage{longtable}
\usepackage{color}

\numberwithin{equation}{section}

\newtheorem{thm}{Theorem}[section]
\newtheorem{cor}[thm]{Corollary}
\newtheorem{lem}[thm]{Lemma}
\newtheorem{prop}[thm]{Proposition}
\newtheorem{defn}[thm]{Definition}

\newtheorem{ex}[thm]{Example}

\newcommand{\x}{\mathbf{x}}
\newcommand{\xxi}{{\boldsymbol{\xi}}}
\newcommand{\zzeta}{{\boldsymbol{\zeta}}}

\newcommand{\N}{\mathbb{N}}
\newcommand{\Z}{\mathbb{Z}}
\newcommand{\Q}{\mathbb{Q}}
\newcommand{\R}{\mathbb{R}}

\newcommand{\Idem}{\mathbb{I}}
\newcommand{\MM}{\mathbb{M}}

\newcommand{\BD}{\mathcal{BD}}

\newcommand{\D}{\mathcal{D}}

\newcommand{\F}{\mathcal{F}}
\newcommand{\U}{\mathcal{U}}
\newcommand{\V}{\mathcal{V}}

\newcommand{\Fin}{\textit{Fin}}

\newcommand{\bN}{\beta\mathbb{N}}

\newcommand{\hN}{{}^*\N}

\newcommand{\hf}{{}^*f}

\newcommand{\ueq}{{\,{\sim}_{{}_{\!\!\!\!\! u}}\;}}

\title{Ramsey properties of 
\\
nonlinear Diophantine equations}

\author{Mauro Di Nasso\thanks
{Supported by PRIN 2012 ``Models and Set'', 
MIUR (Italian Ministry of University and Research).}\\
Department of Mathematics, University of Pisa,\\
Largo Pontecorvo 5, 56127 Pisa, ITALY\\
mauro.di.nasso@unipi.it 
\and Lorenzo Luperi Baglini\thanks{Supported by grant M1876-N35 of the Austrian Science Fund FWF.}\\
University of Vienna, Faculty of Mathematics,\\
Oskar-Morgenstern-Platz 1, 1090 Vienna, AUSTRIA,\\
lorenzo.luperi.baglini@univie.ac.at}
\begin{document}
\date{}
\maketitle

\begin{abstract}
We prove general sufficient and necessary conditions
for the partition regularity of Diophantine equations,
which extend the classic Rado's Theorem by covering
large classes of nonlinear equations. 
Sufficient conditions are obtained by exploiting
algebraic properties in the space of ultrafilters $\beta\N$,
grounding on combinatorial properties of positive density sets
and IP sets. Necessary conditions are proved by a new technique
in nonstandard analysis, based on the use of the relation 
of $u$-equivalence for the hypernatural numbers $\hN$.
\end{abstract}

\section*{Introduction}
Ramsey theory studies structural combinatorial properties
that are preserved under finite partitions. 
An active area of research in this framework
has overlaps with additive number theory,
and it focuses on partition properties of the natural numbers
related to their semiring structure.
Historically, the first result of this kind dates back to 1916;
it is a combinatorial
lemma that I. Schur \cite{S16} used to prove
the existence of non-trivial solutions to Fermat
equations $x^n+y^n=z^n$ modulo $p$ for
all sufficiently large primes $p$.
Precisely, Schur's Lemma 
states that in every finite coloring (partition)
of the natural numbers, one finds a monochromatic 
triple of the form $a,b,a+b$.
Such a property can be phrased by saying that the
equation $x+y=z$ is partition regular on $\N$. 
Another simple equation that is partition regular
is $x+y=2z$; indeed, this amounts to saying that
in every finite coloring of $\N$ one finds
a $3$-term monochromatic arithmetic progression $a, a+d, a+2d$.
(We recall that by van der Waerden's Theorem \cite{vdW27},
another classic result in Ramsey theory that was proved in 1927,
in every finite coloring of $\N$ one actually finds 
arbitrarily long monochromatic arithmetic progressions.)
However, simple examples of equations that are not
partition regular are easily found; \emph{e.g.} $x+y=3z$.

In 1933, R. Rado \cite{R33} completely
characterized partition regular 
systems of linear Diophantine equations on $\N$,
by isolating a simple sufficient and necessary 
condition on the coefficients, the so-called
\emph{column property}. Here is the
formulation for a single equation.\footnote
{~%There is a more general formulation for systems of equations. 
For a full treatment of Rado's Theorem,
see \S 3.2 and \S 3.3 of \cite{GRS90}. }

\smallskip
\noindent
\textbf{Rado's Theorem.}
\emph{A linear Diophantine equation with no constant term
$$c_{1}x_{1}+\dots+c_{n}x_{n}=0$$
is partition regular on $\N$ if and only if the following condition
%``Rado's condition'' 
is satisfied:
\begin{itemize}
\item
``There exists a nonempty set $J\subseteq\{1,\dots,n\}$ such that
$\sum_{j\in J} c_{j}=0$.''
\end{itemize}}

Over the years, an active research focused on 
possible extensions of Rado's Theorem in several directions. 
A large amount of interesting
results have been obtained during the last twenty years
about the various aspects of partition regularity of 
finite and infinite systems of linear equations 
(see, \emph{e.g.}, \cite{BHLS15,BJM15,DHLS14,DHLL95,
GHL14,H07,HL93,HL06,HLS02,HLS03,HLS15,LR07,SV14}).
However, progress on the nonlinear case has been scarce,
and structural theorems that provide an 
overall understanding of Ramsey properties
of nonlinear Diophantine equations are still missing.

Let us briefly recall all the relevant results on this topic that
we are aware of. The simplest result is the multiplicative formulation of Rado's Theorem.

\smallskip
\noindent
\textbf{Multiplicative Rado's Theorem.}\emph{A nonlinear Diophantine equation of the form 
\begin{center} $\prod\limits_{i=1}^{n} x_{i}^{c_{i}}=1$ \end{center}
is partition regular on $\N$ if and only if the following condition is satisfied:
\begin{itemize}
	\item ``There exists a nonempty set $J\subseteq\{1,\dots,n\}$ such that
$\sum_{j\in J} c_{j}=0$.''
\end{itemize}}

The first attempt for a systematic study of the nonlinear case is found
in the paper \cite{L91} of 1991, where H. Lefmann 
characterized the partition regularity of systems of homogeneous
polynomials where every monomial
contains a single variable raised to the same exponent 
$1/k$, and where different equations have different variables. 
Here is the formulation for a single equation.

\smallskip
\noindent
\textbf{Lefmann's Theorem.}
\emph{Let $k\in\N$. A Diophantine equation of the form
$$c_{1}x_{1}^{1/k}+\dots+c_{n}x_{n}^{1/k}=0$$
is partition regular on $\N$ if and only if  ``Rado's condition'' 
is satisfied:
\begin{itemize}
\item
``There exists a nonempty set $J\subseteq\{1,\dots,n\}$ such that
$\sum_{j\in J} c_{j}=0$.''
\end{itemize}}

For instance, a consequence of Lefmann's Theorem
is that the analog of Schur's Lemma
for reciprocals is valid, \emph{i.e.}
the equation $1/x+1/y=1/z$ is partition regular.

Most of the research on the partition regularity of nonlinear Diophantine equations has been done in the past 10 years.

In 2006,  A. Khalfalah and E. Szemer\'edi \cite{KS06}
proved that if $P(z)\in\Z[z]$ takes even values
on some integer, then the equation $x+y=P(z)$ 
is ``partially'' partition regular in the variables $x$ and $y$,
\emph{i.e.}, for every finite coloring of $\N$ one finds a solution 
$x,y,z$ where $x$ and $y$ are monochromatic.

In the paper \cite{CGS12} appeared in
2012 (but whose first draft circulated since 2010), 
P. Csikv\'ari, K. Gyarmati and A. S\'ark\"ozy 
proved a few density results involving nonlinear problems over $\N$ and over finite fields. In particular, they proved that the equation $x+y=z^2$ is not PR.\footnote
{~In that paper there is also a proof of the partition regularity
of $x(y+z)=yz$, which is the same as the 
``reciprocal Schur-equation'' 
%$1/x+1/y=1/z$, already considered in \cite{L91}, as 
mentioned above, and a proof of the partition regularity of $xy=z^{2}$, which is a particular case of the multiplicative Rado's Theorem.}
At the foot of the paper, they left as an open problem 
the partition regularity of $x+y=tz$, which is particularly
relevant as the most basic equation that mixes additive and multiplicative
structure on $\N$.
In 2011, by using algebra in the space of ultrafilters $\beta\N$,
N. Hindman \cite{H11} solved that problem in the positive,
by showing the partition regularity of all
equations of the form $\sum_{i=1}^n x_i=\prod_{i=1}^n y_i$.
In 2014, the second named author \cite{LB14} extended
Hindman's result, and by nonstandard methods
he proved the following:
For every choice of sets $F_i\subseteq\{1,\ldots,m\}$,
the equation $\sum_{=1}^n c_i\, x_i(\prod_{j\in F_i}y_j)=0$ 
is partition regular whenever 
%$\sum_{i=1}^n c_ix_i$
%satisfies ``Rado's condition'', \emph{i.e.}, 
$\sum_{i\in J}c_j=0$ for some nonempty 
$J\subseteq\{1,\ldots,m\}$. 
(It is agreed that $\prod_{j\in\emptyset}y_j=1$.)

An important contribution in the case of quadratic equations 
has been recently given by N. Frantzikinakis and B. Host. 
As a consequence of their structural theorem for multiplicative 
functions \cite{FH16}, they proved that the  equations 
$16x^2+9 y^2=z^2$ and $x^2-xy+y^2=z^2$ are partially
partition regular in the variables $x$ and $y$.

To our knowledge, the last progress done in this 
area about is found in \cite{DNR16}, where 
M. Riggio and the first named author used nonstandard analysis
to identify a large class of Fermat-like equations that
are not partition regular, the simplest examples
being $x^m+y^n=z^k$ where $k\notin\{n,m\}$.\footnote
{~Here we do not count the constant solution 
$z=y=z=2$ of $x^n+y^n=z^{n+1}$.}

At the moment this paper was completed, it was breaking
news that M. J. H. Heule, O. Kullmann and V. W. Marek \cite{HKM16}
solved a problem posed by P. Erd\H{o}s and R. Graham in the 1970s, 
namely the \emph{Boolean Pythagorean triples problem},
that asked whether the equation 
$x^{2}+y^{2}=z^{2}$ is partition regular for 2-colorings of $\N$. 
By using a computer-assisted proof, they have been able to 
prove that any 2-coloring of 
$\{1,2,\dots,7825\}$ contains a monochromatic Pythagorean triple, 
and that $7825$ is the least number with such a 
property.\footnote{~See also the article of E. Lamb 
appeared online in the journal \emph{Nature} on May 26, 2016.
It is worth noticing that the proof (contained in a 
huge file of 200 terabytes), does not solve the full problem
of PR of the Pythagorean equation $x^{2}+y^{2}=z^{2}$,
where a finite (but \emph{arbitrary}) number of
colors is allowed in partitions.}

In this paper we consider Diophantine equations
in their full generality, aiming at finding
simple conditions on coefficients and exponents
that characterize partition regularity
and non-partition regularity. The ultimate goal
is to extend Rado's Theorem and develop a general Ramsey theory of Diophantine equations. 

The techniques that are used here are twofold.
On the side of sufficient conditions for partition regularity (Section \ref{Sufficient}),
we use the algebraic structure of the space of ultrafilters $\beta\N$,
combined with properties of difference sets of sets of positive
asymptotic density. On the side of necessary conditions (Section \ref{nonPR}),
we work in the setting of hypernatural numbers $\hN$
of nonstandard analysis, the instrumental tool 
being the relation of $u$-equivalence and its properties.
Basically, $u$-equivalence formalizes the well-known
characterization of partition regularity in terms
of ultrafilters within a nonstandard framework. 
However, whilst this technique is based on
nonstandard analysis, the used arguments are
of a purely combinatorial nature. 

\section{Preliminary definitions and results}

\subsection{Asymptotic density}

Following a common practice in number theory,
with $\N$ we denote the set of \emph{positive integers}. We recall that the \emph{upper asymptotic density}
of a set $A\subseteq\Z$ is defined as follows:
$$\overline{d}(A)=\limsup_{n\to\infty}\frac{|A\cap[-n,n]|}{2n+1}.$$

By replacing symmetric intervals $[-n,n]$ with
arbitrary intervals, one obtains the following generalization.

\begin{defn}
The \emph{Banach density} $\text{BD}(E)$ of a set $E\subseteq\Z^t$ is 
the greatest of the following superior limits of relative densities
$$\limsup_{n\to\infty}\frac{|E\cap R_n|}{|R_n|}$$
where $(R_n=\prod_{i=1}^t[a_{ni},b_{ni}])_{n\in\N}$ 
are sequences of rectangles whose size
in every direction approaches infinity, \emph{i.e.} 
$\lim_{n\to\infty}(b_{ni}-a_{ni})=+\infty$
for $i=1,\ldots,t$.
\end{defn}

It can be checked that such a greatest value is actually attained.
In the one-dimensional case, equivalently one can define
$\text{BD}(E)=\lim_{n}e_n/n=\inf e_n/n$,
where $e_n$ is the greatest cardinality of an
intersection $E\cap I$ where $I$ is an interval of length $n$.

%Recall that a set $A\subseteq\N$ is \emph{thick} if
%it includes arbitrarily long intervals of arbitrary length.
%It is easily shown that $A$ is thick if and only if $\text{BD}(A)=1$.
%We will use the following well-known property:

\subsection{IP-sets}

A relevant notion in combinatorial number theory 
is that of \emph{IP-set}. 

\begin{defn}
Let $G=(g_i)_{i\in\N}$ be an increasing sequence of natural numbers. 
The \emph{IP-set generated} by $G$ is the set of finite sums
$$\text{FS}(G)=\text{FS}(g_i)_{i\in\N}=
\left\{\sum_{j=1}^k g_{i_{j}}\,\Big|\, 
i_{1}<i_{2}<\dots<i_{k}\right\}.$$
A set $A\subseteq\N$ is called \emph{IP-large} if it contains an IP-set.
\emph{Multiplicative IP-sets} and 
\emph{IP-large sets} are defined similarly: the \emph{multiplicative IP-set generated} by $G$ is
$$\text{FP}(G)=\text{FP}(g_i)_{i\in\N}=
\left\{\prod_{j=1}^k g_{i_{j}}\,\Big|\, 
i_{1}<i_{2}<\dots<i_{k}\right\},$$
and a set $A\subseteq\N$ is called \emph{multiplicatively IP-large} if it contains a multiplicative IP-set.
\end{defn}

By the celebrated \emph{Hindman's Theorem} \cite{H74},
in every finite partition of the natural numbers
$\N=C_1\cup\ldots\cup C_r$, one of the pieces
is additively IP-large; this result can be improved to
obtain the existence of a single piece that is both
additively and multiplicatively IP-large (see \S 5.3 of \cite{HS11}).

An instrumental tool for the main result in this section
is a theorem proved by V. Bergelson, H. Furstenberg and 
R. McCutcheon \cite[Theorem C]{BFMC96}, that we now recall.

Let us first fix a convenient notation.
Let $\Fin$ denote the family of all nonempty finite subsets of $\N$.
Given an increasing sequence 
$G=(g_i)_{i\in\N}$ of natural numbers,
for $\alpha\in\Fin$ denote by $n_\alpha=\sum_{i\in\alpha}g_i$.
Clearly, $n_{\alpha}+n_{\beta}=n_{\alpha\cup\beta}$ whenever 
$\alpha\cap\beta=\emptyset$, and the IP-set $\text{FS}(G)$ is obtained
as the range of the
sequence $(n_{\alpha})_{\alpha\in\Fin}$.
Conversely, if $(n_\alpha)_{\alpha\in\Fin}$ is a
sequence such that $n_{\alpha}+n_{\beta}=n_{\alpha\cup\beta}$
whenever $\alpha\cap\beta=\emptyset$,
then its range is an IP-set, namely
$\{n_\alpha\mid\alpha\in\Fin\}=\text{FS}(G)$ where
$G=(n_{\{i\}})_{i\in\N}$. So, in a precise sense,
the two notions are equivalent.

\begin{thm}[\cite{BFMC96}, Theorem C]\label{C} 
Let $E\subseteq\Z^t$ have positive Banach density, and let

\begin{itemize}
\item
$P_1,\ldots,P_t\in\Z[x_1,\ldots,x_k]$ be polynomials with 
no constant terms\,;
\item
$(n_\alpha^{(1)})_{\alpha\in\F}, \ldots, (n_\alpha^{(k)})_{\alpha\in\F}$ 
be additive IP-sets.
\end{itemize}

Then there exist $e_1,e_2\in E$ and $\alpha\in\Fin$ such that
$$e_1-e_2=\left(P_{1}(n_{\alpha}^{(1)},\dots,n_{\alpha}^{(k)}),\dots,
P_{t}(n_{\alpha}^{(1)},\dots,n_{\alpha}^{(k)})\right).$$
\end{thm}

\subsection{Algebra in the space of ultrafilters $\beta\N$}

In this paper we assume the reader to be familiar with the fundamental 
properties of the space $\bN$ of ultrafilters on $\N$
endowed with the operations of pseudo-sum $\oplus$ 
and pseudo-product $\odot$:
\begin{itemize}
\item
$A\in\U\oplus\V\Leftrightarrow\{n\mid A-n\in\V\}\in\U$,
where $A-n=\{m\in\N\mid m+n\in A\}$;
\item
$A\in\U\odot\V\Leftrightarrow\{n\mid A/n\in\V\}\in\U$,
where $A/n=\{m\in\N\mid mn\in A\}$.
\end{itemize}

In particular, we assume some knowledge of 
\emph{idempotent ultrafilters} and left and right \emph{ideals}
in the compact \emph{topological right semigroups} $(\beta\N,\oplus)$
and $(\beta\N,\odot)$.
For simplicity, we will use the adjective ``additive'' when referring to the
former, and ``multiplicative'' when referring to the latter. 
So, for instance, the ultrafilter
$\U$ is additively idempotent if $\U\oplus\U=\U$, and 
$\U$ is multiplicatively idempotent if $\U\odot\U=\U$.
We will use the following notation. 
\begin{itemize}
\item
$K(\oplus)$ is the minimal additive two sided ideal\,;
\item
$K(\odot)$ is the minimal multiplicative two sided ideal\,;
\item
$\Idem(\oplus)$ is the set of 
additively idempotent ultrafilters\,;
\item
$\Idem(\odot)$ is
the set of multiplicatively idempotent ultrafilters\,;
\item
$\MM(\oplus)=\Idem(\oplus)\cap K(\oplus)$ is
the set of minimal additive idempotents\,;
\item
$\MM(\odot)=\Idem(\odot)\cap K(\odot)$ is
the set of minimal multiplicative idempotents\,;
\item
$\BD=\{\U\in\bN\mid \forall A\in\U\ \text{BD}(A)>0\}$;
\item
$\D=\{\U\in\bN\mid \forall A\in\U\ \overline{d}(A)>0\}\subseteq\BD$.
\end{itemize}

For convenience, we itemize the known results 
about algebra in $\beta\N$ that 
we will use in this paper. A comprehensive reference 
is Hindman and Strauss' book \cite{HS11},
where all proofs can be found.\footnote
{~Precisely, property (B1) and (B3) are particular cases of
Lemma 5.11;  property (B2) is
Theorem 5.20; property (B4) is Theorem 5.20. A
proof of properties (B5) and (B6) is found 
in \S 20.1, where $\BD$ is denoted 
$\Delta^*(\N,+)$; 
and properties (B7) and (B8) are in Theorem 6.79,
where $\D$ is denoted 
$\Delta(\N,+)$. Finally,
property (B9) is Lemma 17.2
where our set $\overline{\MM(\oplus)}$ is 
denoted $\MM$.}

\begin{enumerate}
\item[(B1)]
The closure 
$\overline{\Idem(\oplus)}=
\{\U\in\bN\mid \forall A\in\U\ A\ \text{is IP-large}\}$\,;
\item[(B2)]
$\overline{\Idem(\oplus)}$ is a multiplicative left ideal\,;
\item[(B3)]
The closure 
$\overline{\Idem(\odot)}=
\{\U\in\bN\mid \forall A\in\U\ A\ \text{is multiplicative IP-large}\}$;
\item[(B4)]
$\overline{\MM(\oplus)}$ is a multiplicative left ideal\,;
\item[(B5)]
$\BD$ is a closed additive
two sided ideal\,;
\item[(B6)]
$\BD$ is a closed multiplicative left ideal\,;
\item[(B7)]
$\D$ is a closed additive
left ideal\,;
\item[(B8)]
$\D$ is a closed multiplicative left ideal\,;
\item[(B9)]
$\D\cap\overline{\MM(\oplus)}\cap\MM(\odot)\ne\emptyset$.
\end{enumerate}

Ultrafilters in $\D\cap\overline{\MM(\oplus)}\cap\MM(\odot)$
are particularly relevant.
They were first isolated and studied and by
N. Hindman and D. Strauss, who named them
\emph{combinatorially rich}
(\cite[Definition 17.1]{HS11}).

\subsection{Partition regularity of functions}

By \emph{finite coloring}
we mean a finite partition of the natural numbers.
Elements $a_1,\ldots,a_k$ are called \emph{monochromatic}
with respect to a given finite coloring $\N=C_1\cup\ldots\cup C_r$
if there exists $C_i$ such that $a_1,\ldots,a_k\in C_i$.

\begin{defn} 
{\rm A function $f(x_1,\dots,x_n)$ is called
\emph{partition regular} on $\N$ (or simply PR)
if in every finite coloring of $\N$
one finds a monochromatic root,
%to the equation $f(x_1,\ldots,x_n)=0$, 
\emph{i.e.}
monochromatic elements 
$a_1,\ldots,a_n$ with $f(a_1,\dots,a_n)=0$.

When it is possible to find such elements $a_i$ that are
pairwise different,
the function is called \emph{injectively} PR.

More generally, $f(x_1,\ldots,x_n)$ is called partition regular with 
\emph{injectivity}
$|\{x_{i_1},\ldots,x_{i_p}\}|\ge s$ if in every finite coloring
one can always find a monochromatic
solution $a_1,\ldots,a_n$ with $|\{a_{i_1},\ldots,a_{i_p}\}|\ge s$.

A function $f(x_1,\ldots,x_n)$ is called
non-trivially PR if it is partition regular with injectivity $|\{x_1,\ldots,x_n\}|\ge 2$.}
\end{defn}

The above definitions of partition regularity
are extended to equations $f(x_1,\ldots,x_n)=g(y_1,\ldots,y_k)$
in the obvious way, by considering the corresponding notions 
for the function $f-g$.
So, for instance, the classic \emph{Schur's Theorem} \cite{S16} of 1916
can be equivalently formulated as: 
``The function $f(x,y,z)=x+y-z$ is injectively PR'', or as: 
``The equation $x+y=z$ is injectively PR''.

A fundamental result about partition regularity that
dramatically generalizes the result of Shur's mentioned above,
was proved  in 1933 by R. Rado \cite{R33},
who completely solved the linear (homogeneous and inhomogeneous) Diophantine case.

\begin{thm}[Rado] 
A linear Diophantine homogeneous equation 
$$c_{1}x_{1}+\dots+c_{n}x_{n}\ =\ 0$$
is PR on $\N$ if and only if the following
``Rado's condition'' is satisfied:
\begin{itemize}
\item
``There exists a nonempty set $J\subseteq\{1,\dots,n\}$ such that
$\sum_{j\in J} c_{j}=0$.''
\end{itemize}

Moreover, a linear Diophantine inhomogeneous equation 
$$c_{1}x_{1}+\dots+c_{n}x_{n}=d$$
is PR on $\mathbb{N}$ if and only if either 
\begin{itemize}
	\item there exists a natural number $k$ such that $\sum_{i=1}^{n} c_{i} k=d$ or
	\item there exists an integer $z$ such that $\sum_{i=1}^{n}c_{i} z=d$ and there exists a nonempty subset $J\subseteq\{1,...,n\}$ such that $\sum_{j\in J}c_{j}=0$.
\end{itemize}
\end{thm}

Notice that one cannot have \emph{injective} PR when
the number of variables $n=2$ because, in this case, 
Rado's condition implies that $c_1=-c_2$, and 
the equation reduces to the trivial equality $x_1=x_2$. 
On the other hand, as recently shown by N. Hindman and 
I. Leader in a more general setting, 
the following holds:

\begin{thm}[\cite{HL06}, Theorem 3.1]\label{injectiveRado}
A linear Diophantine equation in more than two variables
is PR on $\N$ if and only if it is injectively PR on $\N$.
\end{thm}

\emph{E.g.}, for $n\geq 2$ the following equations are injectively PR:
$$x_1=x_2+a_1y_1+\ldots+a_ny_n.$$

It is known that partition regularity can be equivalently
expressed in terms of ultrafilters.
(For a precise formulation of this equivalence,
see \cite{HS11}, Theorem 5.7.)
Here we specify this equivalence when restricting to the partition regularity of functions with injectivity conditions.

\begin{defn}
{\rm An ultrafilter $\U$ on $\N$ is called a PR-\emph{witness}
of the function $f(x_1,\ldots,x_n)$
with injectivity $|\{x_{i_1},\ldots,x_{i_p}\}|\ge s$
if for every $A\in\U$ there exist $a_{1},\dots,a_{n}\in A$
such that $f(a_1,\ldots,a_n)=0$ 
and $|\{a_{i_1},\ldots,a_{i_p}\}|\ge s$.}
\end{defn}

\begin{prop}\label{Ultrafilters Characterization} 
A function $f(x_{1},\dots,x_{n})$ is partition regular
with injectivity $|\{x_{i_1},\ldots,x_{i_p}\}|\ge s$
if and only if there exists a PR-witness of $f(x_1,\ldots,x_n)$
with injectivity $|\{x_{i_1},\ldots,x_{i_p}\}|\ge s$.
\end{prop}

\begin{proof}
One direction is trivial because in every finite coloring, 
one and only one of the colors must belong to $\U$, 
by the property of ultrafilter. 
Conversely, notice that the following family
$$\F=\{A\subseteq\N\mid \forall a_1,\ldots,a_n\in A^c\ 
|\{a_{i_1},\ldots,a_{i_p}\}|\ge s\Rightarrow f(a_1,\ldots,a_n)\ne 0\}$$
has the finite intersection property. Indeed, if by
contradiction $A_1,\ldots,A_n\in\F$ are
such that $\bigcap_{i=1}^n A_i=\emptyset$, then
the finite coloring $\N=A_1^c\cup\ldots\cup A_n^c$
would provide a counter-example to the PR
of $f(x_1,\ldots,x_n)$ with injectivity 
$|\{x_{i_1},\ldots,x_{i_p}\}|\ge s$.
Finally, notice that any ultrafilter $\U\supseteq\F$ is 
the desired PR-witness; indeed, if $B\in\U$ was a counter-example,
then its complement $B^c\in\F\subseteq \U$, and hence 
$\emptyset=B\cap B^c\in\U$, a contradiction.
\end{proof}

\section{Sufficient conditions for PR}\label{Sufficient}

Let us first prove a useful property about ultrafilters
that simultaneously witness several equations.

\begin{lem}\label{lemmasystem}
Assume that for every $i=1,\ldots,k$, the ultrafilter $\U$ is a 
PR-witness of $f_i(x_{i,1},\ldots,x_{i,n_i})=0$
with injectivity $|\{x_{i,j_1},\ldots,x_{i,j_{p_i}}\}|\ge s_i$. 
If the functions $f_i$ have 
pairwise disjoint sets of variables\footnote
{~That is, 
$\{x_{i,1},\ldots,x_{i,k_i}\}\cap\{x_{j,1},\ldots,x_{j,k_j}\}=\emptyset$
for $j\ne i$.} then $\U$ is also a PR-witness of the following system
of equations:
$$\begin{cases}
f_i(x_{i,1},\ldots,x_{i,n_i})=0, & i=1,\ldots,k;
\\
x_{1,1}=\ldots=x_{k,1},
\end{cases}$$
with injectivity 
$|\{x_{i,j_1},\ldots,x_{i,j_{p_i}}\}|\ge s_i$ for $i=1,\ldots,k$.
\end{lem}

\begin{proof}
Let $A\in\U$ be fixed. For every $i=1,\ldots,k$ let
$$\begin{aligned}
\Lambda_i\ =\ \{\,a\in A\mid \exists\, a_{i,2},\dots,a_{i,n_i}\in A \ \text{s.t.} \ 
|\{a_{i,j_1},\ldots,a_{i,j_{p_i}}\}|\ge s_i 
\\ 
\&\ f_i(a,a_{i,2},\ldots,a_{i,n_i})=0\,\}.
\end{aligned}$$
Notice that $\Lambda_i\in\U$, as otherwise
$$\begin{aligned}
\Lambda_i^c\cap A\ =\ \{\,a\in A\mid 
\forall\, a_{i,2},\dots,a_{i,n_i}\in A\ 
\ |\{a_{i,j_1},\dots,a_{i,j_{p_i}}\}|\ge s_i\Rightarrow 
\\
f_i(a,a_{i,2},\ldots,a_{i,n_i})\ne 0\,\}
\end{aligned}$$
would belong to $\U$, 
contradicting the hypothesis that $\U$ is a witness of $f_i$
with injectivity $|\{x_{i,j_1},\ldots,x_{i,j_{p_i}}\}|\ge s_i$.
Then the intersection $\Lambda=\bigcap_{i=1}^n\Lambda_i\in\U$
is nonempty and we can pick an element $a\in\Lambda\subseteq A$.
It directly follows from the definitions that there are
elements $a_{i,2},\ldots,a_{i,n_i}\in A$ with 
$|\{a_{i,j_1},\ldots,a_{i,j_{p_i}}\}|\ge s_i$
and such that $f_i(a,a_{i,2},\ldots,a_{i,n_i})=0$
for $i=1,\ldots,k$. This shows the existence of
solutions in $A$ to the considered system, with the claimed injectivity properties.
\end{proof}

Recall that a function $f(x_1,\ldots,x_n)$ is 
\emph{homogeneous}, if there exists
$\ell$ such that for every $\lambda,x_{1},\dots,x_{n}$ one has 
$f(\lambda x_1,\ldots,\lambda x_n)=\lambda^\ell f(x_1,\ldots,x_n)$.
In this case $\ell$ is called the degree of homogeneity of $f$.

The following ultrafilter property was first proved 
by the second-named author by
nonstandard analysis;
the proof given below uses an 
essentially equivalent ultrafilter argument.

\begin{thm}[\cite{LB14}]\label{alpha} 
Assume that the equation 
$f(x_{1},\dots,x_{n})$ is PR 
with injectivity $|\{x_{i_1},\ldots,x_{i_p}\}|\ge s$.
If $f$ is homogeneous then the set of witnesses
$$\mathfrak{W}_f= 
\{\U\in\bN\mid \U \ \text{is a witness of}\ f\ 
\text{with injectivity}\ 
|\{x_{i_1},\ldots,x_{i_p}\}|\ge s \}$$
is a closed multiplicative two sided ideal.
\end{thm}

\begin{proof} 
Let $\U\in\mathfrak{W}_f$ and $\V\in\bN$. 
By definition, $B\in\U\odot\V$ if and only if
$\widehat{B}=\{m\in\N\mid B/m\in\V\}\in\U$.
Now let $b_{1},\ldots,b_{n}\in\widehat{B}$ be
such that $f(b_1,\ldots,b_n)=0$ and $|\{b_{i_1},\ldots,b_{i_p}\}|\ge s$,
pick any $\lambda\in\bigcap_{j=1}^kB/b_{i_j}\in\V$,
and consider the elements $\lambda b_1,\ldots,\lambda b_n\in B$.
By homogeneity, 
$f(\lambda b_1,\ldots,\lambda b_n)=\lambda^\ell f(b_1,\ldots,b_n)=0$;
moreover, $|\{\lambda b_{i_1},\ldots,\lambda b_{i_p}\}|=
|\{b_{i_1},\ldots,b_{i_p}\}|\ge s$.
This shows that $\U\odot\V\in\mathfrak{W}_f$,
and hence we can conclude that $\mathfrak{W}_f$
is a multiplicative right ideal. 
Moreover, it is verified in a straightforward manner
that $\mathfrak{W}_f$ is (topologically) closed in $\beta\N$.
Finally, recall that every closed right ideal is also a left ideal
(see \cite[Theorem 2.19]{HS11}).
\end{proof}

The intersection of all closed two sided ideals 
equals the closure of the minimal ideal, and so:

\begin{cor}\label{ultrahomogeneous}
Let $f(x_1,\ldots,x_n)$ be a homogeneous function
that is PR with injectivity $|\{x_{i_1},\ldots,x_{i_p}\}|\ge s$.
Then every $\U\in\overline{K(\odot)}$ is a witness 
of $f$ with injectivity $|\{x_{i_1},\ldots,x_{i_p}\}|\ge s$. 
\end{cor}

Next we give an ultrafilter proof of a result by 
T.C. Brown and V. R\"odl, showing that
the class of PR of homogeneous functions 
is stable under the operation of ``inverting variables''.

\begin{thm}[\cite{BR91}]\label{inverses}
If a homogeneous function $f(x_1,\ldots,x_n)$
is PR with injectivity $|\{x_{i_1},\ldots,x_{i_p}\}|\ge s$
then also $f(1/x_1,\ldots,1/x_n)$ is PR
with injectivity $|\{x_{i_1},\ldots,x_{i_p}\}|\ge s$.
\end{thm}

\begin{proof}
Pick a non-principal ultrafilter $\U$ on $\N$ that is a PR-witness of $f$.
Let $\rho:\N\to\Q$ be the ``reciprocal map'' 
$\rho:\N\to\Q$ where $\rho(n)=1/n$, and let
$\varphi:\N\to\Q$ be ``factorial map'' where $\varphi(n)=n!$.
Then consider the \emph{image ultrafilters} $\U_1=\rho(\U)$ and 
$\U_2=\varphi(\U)$.\footnote
{~Recall that if $\U$ is an ultrafilter on a set $I$
and $f:I\to J$ is a function, the image $f(\U)$ is the ultrafilter on $J$
where $A\in f(\U)\Leftrightarrow \{i\in I\mid f(i)\in A\}\in\U$
for every $A\subseteq J$.}
Since $\U_1,\U_2$ are ultrafilters on $\Q$,
it makes sense to consider their pseudo-product
$\V=\U_1\odot\U_2$ on $\Q$, which is defined similarly
as pseudo-products in $\beta\N$. We want to show that
$\N\in\V$, and the ultrafilter $\V_\N=\V\cap\mathcal{P}(\N)$
is a PR-witness of $f(1/x_1,\ldots,1/x_n)$.
By the definitions,
$$\N\in\V\ \Leftrightarrow\ \{u\in\Q\mid \N/u\in\U_2\}\in\U_1\ \Leftrightarrow\
\Lambda=\{n\in\N\mid \N/\!{}_{1/n}\in\U_2\}\in\U.$$
For every $n\in\N$, we have that
$$\N/\!{}_{1/n}\in\U_2\ \Leftrightarrow\
\Gamma_n=\{m\in\N\mid m!\in\N/\!{}_{1/n}\}=
\{m\in\N\mid m!/n\in\N\}\in\U.$$
Notice that $\Gamma_n\in\U$ because it contains
all $m\ge n$ and $\U$ is non-principal. So, $\Lambda=\N\in\U$,
and this proves that $\N\in\V$. 
Now let $B\in\V_\N$. Since $B\in\V$, the
set $\Lambda(B)=\{n\in\N\mid B/\!{}_{1/n}\in\U_2\}\in\U$, and hence
there exist $a_1,\ldots,a_n\in \Lambda(B)$
such that $|\{a_{i_1},\ldots,a_{i_p}\}|\ge s$ and $f(a_1,\ldots,a_n)=0$.
Now recall that $B/{}_{1/a_i}\in\U_2\Leftrightarrow 
\Gamma_i(B)=\{m\in\N\mid k!/a_i\in B\}\in\U$; in particular
we can pick $k\in\bigcap_{i=1}^n\Gamma_i(B)\in\U$.
Finally, notice that elements $b_i=k/a_i\in B$ are such that
$|\{b_{i_1},\ldots,b_{i_p}\}|\ge s$ and
$$f(1/b_1,\ldots,1/b_n)\ =\ 
f(a_1/k,\ldots,a_n/k)\ =\ 
1/k^\ell\cdot f(a_1,\ldots,a_n)\ =\ 0,$$
where $\ell$ is the degree of homogeneity of $f$.
\end{proof}

We are now ready to extend Rado's Theorem
on the side of sufficient conditions for PR.
Let us start with the following consequence of Theorem \ref{C},
which is particularly relevant to our purposes.

\begin{thm}\label{beta} 
Let $c\,(x_1-x_2)=P(y_{1},\dots,y_{k})$ be a Diophantine
equation where the polynomial $P$ has no constant term.
If the set $A\subseteq\N$ is IP-large and has positive Banach density
then there exist $\xi_1,\xi_2\in A$ and mutually distinct 
$\eta_{1},\dots,\eta_{k}\in A$ such that 
$c\,(\xi_1-\xi_2)=P(\eta_{1},\dots,\eta_{k})$.
Moreover, if $k=1$ then one can take $\xi_1\neq\xi_2$.
\end{thm}

\begin{proof} 
First of all, notice that we can pick an IP-set 
$\left(n_{\alpha}\right)_{\alpha\in\mathcal{F}}\subseteq A$ 
such that $n_{\alpha}\neq n_{\beta}$ for $\alpha\neq\beta$. 
Indeed, given any IP-set $\text{FS}(g_i)_{i\in\N}$,
one can inductively define a sub-IP-set 
$\text{FS}(g'_i)\subseteq\text{FS}(g_i)$
with the desired property, by setting
$g'_1=g_1$ and $g_{i+1}=\min\{g_j\mid g_j>g'_1+\ldots+g'_i\}$.\par

Now fix a permutation $\sigma:\N\rightarrow\N$ with no finite cycles, 
\emph{i.e.} such that for every $s\in\N$, the iterated composition
$\sigma^{s}(n)\neq n$ for all $n\in\N$. 
This ensures that for every $\alpha\in\F$, the sets
$\alpha_{s}=\{\sigma^{s}(i)\mid i\in\alpha\}$
are pairwise distinct. Indeed, if $\alpha_{s+\ell}=\alpha_s$
for some $s,\ell\in\N$, then $\alpha_\ell=\alpha$ and
$\sigma^\ell$ would have a finite cycle, contradicting our assumption on $\sigma$.
In consequence, by our choice of the IP-set, we have
$n_{\alpha_s}\ne n_{\alpha_{s'}}$ for $s\ne s'$.
Moreover, for every $s$, 
the set $\{n_{\alpha_s}\mid\alpha\in\F\}$
is an IP-set, because for every $\alpha,\beta\in\F$
one has that $n_{{(\alpha\cup\beta)}_s}=n_{\alpha_s}+n_{\beta_s}$
whenever $\alpha\cap\beta=\emptyset$. (Notice that
$(\alpha\cup\beta)_s=\sigma^{s}(\alpha\cup\beta)=
\sigma^{s}(\alpha)\cup\sigma^{s}(\beta)=
\alpha_s\cup\beta_s$, where $\alpha_s\cap\beta_s=\emptyset$ because 
$\alpha\cap\beta=\emptyset$.)
 
Now consider the set $cA=\{c a\mid a\in A\}$. As
$\text{BD}(cA)=\frac{\text{BD}(A)}{|c|}>0$, we can
apply Theorem \ref{C} with $t=1$, $E=cA$,
$P_{1}=P(y_1,\ldots,y_k)$, and 
the IP-sets $(n_{\alpha}^{(s)})$ where $n^{(s)}_\alpha=n_{\alpha_s}$
for $s=1,\ldots,k$. We obtain the existence of elements
$x_1=c\,\xi_1$, $x_2=c\,\xi_2$ where $\xi_1,\xi_2\in A$,
and of numbers 
$\eta_{1}=n_{\alpha}^{(1)},\dots,\eta_{k}=n_{\alpha}^{(k)}$ 
such that 
$$x_1-x_2=c\cdot (\xi_1-\xi_2)=P(\eta_{1},\dots,\eta_{k}).$$
By our definition of the IP-sets $(n_{\alpha}^{(s)})$, the elements 
$\eta_{1},\dots,\eta_{k}$ are mutually distinct.
Finally, notice that when $k=1$ the polynomial $P(y_1)$ can
only have finitely many roots, and so the above arguments 
also apply to $A'=A\setminus\{\text{roots of P}\}$,
which is still an IP-large set with positive Banach density.\footnote
{~This argument 
does not apply to the general case $k>1$; indeed, while 
$A'$ still has positive Banach density, it may no more be 
additively IP-large.}
Clearly, in this case $\xi_1\ne\xi_2$ because
$\xi_1-\xi_2=P(\eta_1)\ne 0$.
\end{proof}

We can now isolate a simple sufficient condition
for a Diophantine nonlinear equation to be PR.

\begin{defn}
A polynomial with integers coefficients is
called a \emph{Rado polynomial} if it can be written in the form
$$c_{1}x_{1}+\dots+c_{n}x_{n}+P(y_{1},\dots,y_{k})$$
where $n\ge 2$, $P$ has no constant term,
and there exists a nonempty subset $J\subseteq\{1,\dots,n\}$ such that
$\sum_{j\in J} c_{j}=0$.\footnote
{~It is assumed that the sets of variables 
$\{x_1,\ldots,x_n\}$ and $\{y_1,\ldots,y_k\}$ are disjoint.}
\end{defn}

Notice that, by Rado's Theorem, a \emph{linear} polynomial 
with integer coefficients is PR if and only if it is a Rado polynomial.
We now show that one implication
in Rado's theorem (namely, that every Rado polynomial is PR with certain injectivity conditions) can be extended to \emph{all}
Rado polynomials. 

\begin{thm}\label{Generalized Rado} 
Let
$$R(x_1,\ldots,x_n,y_1,\ldots,y_k)=
c_1x_1+\ldots+c_n x_n+P(y_1,\ldots,y_k)$$
be a Rado polynomial. 
Then every ultrafilter 
$\U\in\overline{K(\odot)}\cap\overline{\Idem(\oplus)}\cap\BD$
is a PR-witness of $R$ with injectivity 
$|\{x_1,\dots,x_n\}|\geq n-1$ and $|\{y_1,\dots,y_k\}|=k$. 

When $n=2$, every 
$\U\in\overline{\Idem(\oplus)}\cap\BD$
satisfies the above property, and one has
injectivity $|\{x_1,x_2\}|=2$ if $k=1$.
Moreover, if $P\ne 0$ is linear then every $\U\in\overline{K(\odot)}$
is a PR-witness of $R$ with full 
injectivity $|\{x_1,\ldots,x_n,y_1,\ldots,y_k\}|=n+k$.
\end{thm}

Notice that the set
$\overline{K(\odot)}\cap\overline{\Idem(\oplus)}\cap\BD$
is nonempty; indeed, it contains all 
\emph{combinatorially rich ultrafilters}.

\begin{proof} 
Assume first that $n\geq 3$, and consider the following system:
$$\begin{cases}
c_1 z+c_2x_2+\ldots+c_{n}x_{n}=0;
\\
c_1(w-x_1)=P(y_1,\ldots,y_k);
\\
z=w.
\end{cases}$$

The first equation is injectively PR
by Theorem \ref{injectiveRado} and, since it is
homogeneous, it is witnessed by any 
$\U\in\overline{K(\odot)}$, by Corollary \ref{ultrahomogeneous}.
Moreover, if $\U\in\overline{\Idem(\oplus)}\cap\BD$,
every $A\in\U$ is additively IP-large and has positive Banach density
and so, by Theorem \ref{beta}, the second equation
is witnessed by $\U$ with injectivity $|\{y_1,\ldots,y_k\}|=k$.
Then, by Lemma \ref{lemmasystem}, 
every $\U\in\overline{K(\odot)}\cap\overline{\Idem(\oplus)}\cap\BD$ 
is a witness of the above system
with injectivity $|\{z,x_2,\ldots,x_n\}|=n$
and $|\{y_1,\ldots,y_k\}|=k$.
By combining, we finally obtain that $\U$ is a witness of the equation
$$c_1x_1+c_2 x_2+\ldots+c_n x_n+P(y_1,\ldots,y_k)=0$$
with the desired injectivity properties.

When $n=2$, by the hypothesis of Rado polynomial, one has that
$c_1=-c_2=c$. In this case, the equation $R=0$ reduces
to the equation in Theorem \ref{beta}, that
applies to every set $A\in\U\in\overline{\Idem(\oplus)}\cap\BD$.

If $P\ne 0$ is linear, then the injective PR of $R$
is given by Theorem \ref{injectiveRado}.
In this linear case, $R$ is trivially homogeneous and so
every $\U\in\overline{K(\odot)}$ is a witness, 
by Corollary \ref{ultrahomogeneous}.
\end{proof}

We are finally ready to state the following theorem,
that further extends the class of nonlinear polynomials
proved to be PR.

\begin{thm}\label{main}
Let $\mathfrak{F}$ be the family of functions 
whose PR on $\N$ is witnessed by at least an ultrafilter  
$\U\in\overline{\MM(\odot)}\cap\overline{\Idem(\oplus)}\cap\BD$.
Then $\mathfrak{F}$ includes:

\begin{enumerate}
\item
Every Rado's polynomial
$$R(x_1,\ldots,x_n,y_1,\ldots,y_k)=c_1x_1+\ldots+c_n x_n+P(y_1,\ldots,y_k)$$
with injectivity $|\{x_1,\ldots,x_n\}|\ge n-1$ and
$|\{y_1,\ldots,y_k\}|=k$, and 
with injectivity $|\{x_1,x_2\}|=2$
when $n=2$ and $k=1$, and with
full injectivity $|\{x_1,\ldots,x_n,y_1,\ldots,y_k\}|=n+k$ when 
$P\ne 0$ is linear\,;

\item
Every function $f$ of the form
$$f(x,y_1,\ldots,y_k)=x-\prod_{i=1}^k y_i$$
with full injectivity $|\{x,y_1,\ldots,y_k\}|=k+1$;

\item
Every function $f$ of the form
$$f(x,y_1,\ldots,y_k)=x-\prod_{i=1}^k y_i^{a_i}$$
with full injectivity $|\{x,y_1,\ldots,y_k\}|=k+1$, whenever 
the exponents $a_i\in\Z$ satisfy $\sum_{i=1}^n a_i=1$.
\end{enumerate}

\noindent
Moreover, the family $\mathfrak{F}$ satisfies the following closure property:
\begin{enumerate}
\item[4.]
Assume  that $f(z,y_1,\ldots,y_k)\in\mathfrak{F}$ with injectivity
$|\{y_{i_1},\ldots,y_{i_p}\}|\ge s$ and that
%the function
$z-g(x_1,\ldots,x_n)\in\mathfrak{F}$
with injectivity $|\{x_{j_1},\ldots,x_{j_q}\}|\ge t$.
Then $f(g(x_1,\ldots,x_n),y_1,\ldots,y_k)\in\mathfrak{F}$
with injectivity $|\{x_{j_1},\ldots,x_{j_q}\}|\ge t-1$
and $|\{y_{i_1},\ldots,y_{i_p}\}|\ge s-1$.

\item[5.]
Assume that the homogeneous function $f(x_1,\ldots,x_n)\in\mathfrak{F}$
with injectivity $|\{x_{i_1},\ldots,x_{i_p}\}|\ge s$.
Then $f(1/x_1,\ldots,1/x_n)\in\mathfrak{F}$
with injectivity $|\{x_{i_1},\ldots,x_{i_p}\}|\ge s$.
\end{enumerate}
\end{thm}

\begin{proof}
(1). This is Theorem \ref{Generalized Rado}.

(2). Since $\U\in\overline{\Idem(\odot)}$, 
every $A\in\U$ is multiplicatively IP-large,
and so it contains injective
solutions to the equation $x=\prod_{i=1}^n y_i$.

(3). Notice first that 
$\prod_{i=1}^n y_i^{a_i}=x$ is injectively PR.
Indeed, given a finite coloring $\N=C_1\cup\ldots\cup C_r$,
one considers the partition as given
by the sets $D_s=\{n\mid 2^n\in C_i\}$. 
By Theorem \ref{injectiveRado},
one finds pairwise distinct monochromatic
$\eta_1,\ldots,\eta_n,\xi\in D_s$ such that
$\sum_{i=1}^n a_i \eta_i=\xi$. Then
$2^{\eta_1},\ldots,2^{\eta_n},2^\xi\in C_s$ are an injective
monochromatic solution of $\prod_{i=1}^n y_i^{a_i}=x$.
Now, the function $f(x,y_1,\ldots,y_k)=x-\prod_{i=1}^k y_i^{a_i}$
is homogeneous since $\sum_{i=1}^n a_i=1$
and so, by Corollary \ref{ultrahomogeneous}, every ultrafilter 
$\U\in\overline{\MM(\odot)}\subseteq\overline{K(\odot)}$ is 
an injective witness.

(4). It directly follows from Lemma \ref{lemmasystem}. 

(5). By Theorem \ref{inverses}, 
we have that $f(1/x_1,\ldots,1/x_n)$
is PR with injectivity $|\{x_{j_1},\ldots,x_{j_q}\}|\ge t$. 
Since the function is homogeneous, such a PR is witness by all
ultrafilters $\U\in\overline{K(\odot)}$, by Corollary \ref{ultrahomogeneous}.
\end{proof}

Let us now give some examples of equations whose PR
is obtained by applying Theorem \ref{main}.

\begin{ex}\label{ex-one}
Consider the injectively PR polynomials 
$x_{1}x_{2}=z^{2}$, $y_{1}+y_{2}=y_{3}$ and 
$t_{1}-t_{2}=t_{3}$, which are in $\mathfrak{F}$.
Then, by the closure property (4), it follows that the 
following equations are PR with full injectivity. 

\begin{itemize}
\item
$x(y_{1}+y_{2})=z^{2}$,
\item
$x(t_{1}-t_{2})=z^{2}$,
\item
$x_{1}x_{2}=(y_{1}+y_{2})^{2}$,
\item
$x_{1}x_{2}=(t_{1}-t_{2})^{2}$,
\item
$x(t_{1}-t_{2})=(y_{1}+y_{2})^{2}$,
\item
$x(y_{1}+y_{2})=(t_{1}-t_{2})^{2}$,
\item
$(y_{1}+y_{2})(t_{1}-t_{2})=z^{2}$,
\end{itemize}
\end{ex}

\begin{ex}\label{ex-two}
{\rm The example above generalizes as follows: 
Let $n,m\in\N$ and assume that, for every $i\leq n$, $j\leq m$,
the equations 
$$x_{i,1}=\sum_{h=1}^{r_{i}}c_{i,h}x_{i,h}, \ y_{j,1}=\sum_{k=1}^{s_{j}}d_{j,k}y_{j,k}$$
are PR. Let $a_{1},\dots,a_{n},b_{1},\dots, b_{m}$ be such that $\sum_{i=1}^{n}a_{i}=\sum_{j=1}^{m}b_{j}$ and consider the homogeneous PR equation $\prod_{i=1}^{n}t_{i}^{a_{i}}=\prod_{j=1}^{m}z^{b_{j}}$. All these equations are PR and homogeneous and therefore,
by the closure property (4), also
$$\prod_{i=1}^{n}\left(\sum_{h=1}^{r_{i}}c_{i,h}x_{i,h}\right)^{a_{i}}=\ \prod_{j=1}^{m}\left(\sum_{k=1}^{s_{j}}d_{j,k}y_{j,k}\right)^{b_{j}}$$
is PR with full injectivity.}
\end{ex}

\begin{ex}\label{ex-three}
{\rm Notice that all the equations considered in the previous examples
are homogeneous. Therefore by the closure property (5) applied to some of 
the equations of Example \ref{ex-one} we obtain, \emph{e.g.}, 
that the following equations are PR with full injectivity.}

\begin{itemize}
\item
$x^{2}y_{1}y_{2}=z^{2}(y_{1}+y_{2})$\,;
\item
$x^{2}t_{1}t_{2}=z^{2}(t_{2}-t_{1})$\,;
\item
$(y_{1}+y_{2})(t_{2}-t_{1})z^{2}=y_{1}y_{2}t_{1}t_{2}$.
\end{itemize}
\end{ex}

\begin{ex}\label{ex-four}
{\rm Consider the injectively PR polynomials $z_{1}-\sum_{i=1}^n x_i$ and
$z_{2}-\prod_{j=1}^k y_j$, which are in $\mathfrak{F}$.
Then, by the closure property, the equation 
$$\sum_{i=1}^n x_i=\prod_{j=1}^k y_j$$
is PR with injectivity $|\{x_1,\ldots,x_n\}|=n$ and $|\{y_1,\ldots,y_k\}|=k$.
(This result was first proved by N. Hindman in \cite{H11}.)}
\end{ex}

\begin{ex}\label{ex-five}
{\rm Let $R(x_{1},\dots,x_{n})=\sum_{i=1}^{n}c_{i}x_{i}$ be a linear Rado polynomial with $n\geq 3$ and for every $i=1,\dots,n$, 
let $k_{i}\geq 2$ be a given natural number. Consider the polynomials 
$$S_{i}(z_{i},y_{i,1},\dots,y_{i,k_{i}}):=z_{i}-\prod_{j=1}^{k_{i}}y_{i,j}.$$ 
Notice that 
$R(x_{1},\dots,x_{n})$ and $S_{i}(z_{i},y_{i,1},\dots,y_{i,k_{i}})\in\mathfrak{F}$ with full injectivity.
Then we can apply the closure property of $\mathfrak{F}$ 
to the system
$$\begin{cases}
R(x_{1},\dots,x_{n})=0;\\
S_{i}(z_{i},y_{i,1},\dots,y_{i,k_{i}}), & i=1,\ldots,n;
\\
x_{i}=z_{i}, & i=1,\ldots,n.
\end{cases}$$ 
It follows that the equation
$$\sum_{i=1}^{n} \left(c_{i}\prod_{j=1}^{k_{i}}y_{i,j}\right)=0$$
is in $\mathfrak{F}$ with full injectivity. 
(This result was proved, in a more general form, in \cite[Theorem 3.3]{LB14}.)}
\end{ex}

\begin{ex}\label{ex-six}
{\rm For every $n\in\mathbb{N}$, the function $u-v-z^{n}$ is in 
$\mathfrak{F}$ with full injectivity; moreover, for every $k\geq 2$ 
the function $x=\prod_{j=1}^{k}x_{j}$ is in $\mathfrak{F}$ with full injectivity. Therefore, for every $h,k\geq 2$ we can apply the closure property
of $\mathfrak{F}$ to the system
$$\begin{cases}
u-v=z^{n};\\
x=\prod_{j=1}^{h}x_{j};\\
y=\prod_{j=1}^{k}y_{j};\\
x=t; 
\\
y=v.
\end{cases}$$ 
This shows that the equation
$$\prod_{j=1}^{h}x_{j}-\prod_{j=1}^{k}y_{j}=z^{n}$$
is in $\mathfrak{F}$ with full injectivity.\footnote
{~However, as it will be shown in Section \ref{nonPR},
the equation $x^n-y^m=z^k$ is not PR if $n\notin\{m,k\}$.}}
\end{ex}

Let us notice that for $n=h=k=2$, Example \ref{ex-six} reduces to
$$x_{1}x_{2}=y_{1}y_{2}+z^{2}.$$
Such an equation can be considered as a modified version
of the Pythagorean equation $x^{2}=y^{2}+z^{2}$, whose partition regularity 
is one of the most interesting and challenging open problems 
in this field.\footnote
{~As mentioned in the Introduction,
this problem was posed in 1975 by P. Erd\"{o}s and R. Graham 
(see, \emph{e.g.} \cite{G07}). 
The first important progress has been rcently
obtained by M. J. H. Heule, O. Kullmann and V. W. Marek \cite{HKM16},
who proved that every 2-coloring of $\N$ contains a monochromatic 
Pythagorean triple.}

The range of Theorem \ref{main} includes a large family of PR polynomials;
however there exist polynomials that are known to be PR
but do not belong to this family.

\begin{ex}
{\rm The polynomial $P(x_1,x_2,x_3)=x_1x_2-2x_3$
is PR but it does not belong to the family
$\mathfrak{F}$ of Theorem \ref{main}.\footnote
{~Other examples of PR polynomials not included in the 
family $\mathfrak{F}$ of Theorem \ref{main},
can be found in \cite{LB14}.}}
\end{ex}

Indeed, given a finite coloring $\N=C_1\cup\ldots\cup C_r$,
consider the coloring $\N=C'_1\cup\ldots\cup C'_r$
where $C'_i=\{n\in\N\mid 2^n\in C_i\}$.
By the non-homogenous part of Rado's Theorem,
the polynomial  $y_1+y_2-y_3-1$ is PR. 
Let $a,b,c\in C'_i$ be monochromatic numbers 
such that $a+b-c-1=0$. Then $2^a,2^b,2^c\in C_i$ are monochromatic 
solutions $P(2^a,2^b,2^c)=0$.

%Notice that no set of coefficients of $P(x_{1},x_{2},x_{3})$ sums to zero. 
Assume by contradiction that 
$P(x_1,x_2,x_3)=x_1x_2-2x_3\in\mathfrak{F}$.
Notice that the polynomial $x_3-y_1y_2\in\mathfrak{F}$,
and so, by the closure property (4), also
$P(x_1,x_2,y_1y_2):=x_1 x_2-2y_1y_2$ would belong to $\mathfrak{F}$.
This is not possible because $x_1 x_2-2y_1y_2$ is not PR.
To see this, consider the partition $\N=C_1\cup C_2$
where $C_1$ is the set of natural numbers $n$
such that the greatest exponent $k$ with $2^k\,|\, n$ is even.
It is easily verified that $a_1,a_2,b_1,b_2\in C_i\Rightarrow
a_1a_2\ne 2b_1b_2$ for $i=1,2$.

\smallskip
The previous example 
%(as well as Theorem \ref{main}) 
shows that being Rado is not a necessary condition 
for a polynomial to be PR. 
%In the next section we will study this fact in more detail, 
%and focus on necessary conditions for PR.

\section{Necessary conditions for PR}
\label{nonPR}

In this section we isolate necessary conditions
for a Diophantine equation to be PR.
Instead of working in the space of ultrafilters $\beta\N$
as done in the previous section,
here we will use a different, although closely related, 
non-elementary technique, namely \emph{nonstandard analysis}
on the \emph{hypernatural numbers} $\hN$. 
We will assume the reader to be familiar
with the fundamental notions and results of nonstandard analysis,
namely \emph{hyper-extensions} (or nonstandard extensions)
of sets and function, the \emph{transfer} principle, the \emph{overspill}
principle, the $\kappa$-\emph{saturation} property.
All these topics can be found in any of the monographies 
on nonstandard analysis (see, \emph{e.g.}, the books \cite{Go98, LW15}). 

\subsection{$u$-equivalence}

In the following, we will work in a $\mathfrak{c}^+$-saturated extension of $\N$. In addition to the fundamental principles
of nonstandard analysis, our proofs will also use properties 
of the relation of \emph{$u$-equivalence} on hypernatural numbers,  
as introduced by the first named author in \cite{DN15b}.
(See also \cite{DN15c}, where $u$-equivalent pairs are named 
\emph{indiscernible}, and \cite{DN15}, \cite{LB12}, where many algebraic properties of $u$-equivalence are proved by means of iterated hyperextensions.) 

\begin{defn}
Two hypernatural numbers $\xi,\zeta\in\hN$ are
\emph{$u$-equivalent} if they cannot be distinguished 
by any hyperextension, \emph{i.e.}
if for every $A\subseteq\N$ one has
either $\xi,\xi'\in{}^*A$ or $\xi,\xi'\notin{}^*A$.
\end{defn}

The ``$u$'' in $u$-equivalence stands for ``ultrafilter''.
Indeed, to every $\xi\in\hN$ is associated
an ultrafilter $\mathfrak{U}_\xi=\{A\subseteq\N\mid \xi\in{}^*A\}$,
and $\xi\ueq\zeta$ means that the associated ultrafilters coincide:
$\mathfrak{U}_\xi=\mathfrak{U}_\zeta$.

We will use the following properties (see \cite{DN15b}).

\begin{itemize}
\item[(U1)]
If $k\in\N$ is finite and $\xi\ueq k$ then $\xi=k$\,;
\item[(U2)]
For every $f:\N\to\N$, if
$\xi\ueq\zeta$ then ${}^*f(\xi)\ueq{}^* f(\zeta)$\,;
\item[(U3)]
For every $f:\N\to\N$, if ${}^*f(\xi)\ueq\xi$ then ${}^*f(\xi)=\xi$.
\item[(U4)]
$\xi\ueq\zeta$ and $\xi<\zeta$ imply $\zeta-\xi$ infinite.
\end{itemize}

In the language of nonstandard analysis,
we have the following counterpart of 
Proposition \ref{Ultrafilters Characterization} (see also \cite{LB12}, Theorem 2.2.9, which gives a more general version of this result).

\begin{prop}\label{U-equivalence Characterization}
A function $f(x_1,\ldots,x_n)$ is partition regular
with injectivity $|\{x_{i_1},\ldots,x_{i_p}\}|\ge s$
if and only if there exist hypernatural numbers
$\xi_1\ueq\ldots\ueq\xi_n$ with
${}^*f(\xi_1,\ldots,\xi_n)=0$ and $|\{\xi_{i_1},\ldots,\xi_{i_p}\}|\ge s$.
\end{prop}

\begin{proof}
Assume first that there exist 
$\xi_1\ueq\ldots\ueq\xi_n$ with the above properties,
and consider the ultrafilter
$\mathcal{U}=\mathfrak{U}_{\xi_1}=\ldots=\mathfrak{U}_{\xi_n}$.
For every $A\in\mathcal{U}$,
the elements $\xi_i$ witness that the following is true:
$$\exists\, y_1,\ldots,y_n\in{}^*A\ \text{s.t.}\ {}^*f(y_1,\ldots,y_n)=0\ 
\&\ |\{y_{i_1},\ldots,y_{i_p}\}|\ge s.$$
By \emph{transfer}, we obtain the existence of elements 
$a_1,\ldots,a_n\in A$ such that $f(a_1,\ldots,a_n)=0$ and
$|\{a_{i_1},\ldots,a_{i_p}\}|\ge s$, thus showing that
$\U$ is
a PR-witness of $f$ with injectivity $|\{x_{i_1},\ldots,x_{i_p}\}|\ge s$.

Conversely, pick a PR-witness $\mathcal{U}$ of $f$ with
injectivity $|\{x_{i_1},\ldots,x_{i_p}\}|\ge s$.
Then for every $A\in\U$ the following 
set is nonempty:
$$\Gamma(A)\ =\ \left\{(a_1,\ldots,a_n)\in A^n\mid
f(a_1,\ldots,a_n)=0\ \&\ |\{a_{i_1},\ldots,a_{i_p}\}|\ge s\right\}.$$

Since $\Gamma(A_1)\cap\ldots\cap\Gamma(A_k)=
\Gamma(A_1\cap\ldots\cap A_k)$, the family
$\mathfrak{F}=\{\Gamma(A)\mid A\in\U\}$ has the
\emph{finite intersection property}, and hence, 
by $\mathfrak{c}^+$-\emph{enlargement} (which is entailed by $\mathfrak{c}^+$-saturation),
we can pick $(\xi_1,\ldots,\xi_n)\in\bigcap_{A\in\U}{}^*\Gamma(A)$.
It is readily verified that 
$\mathfrak{U}_{\xi_1}=\ldots=\mathfrak{U}_{\xi_n}=\U$,
and that $\xi_1,\ldots,\xi_n$ satisfy the desired properties.
\end{proof}

In particular, we will use the following characterization.

\begin{cor}
A function $f(x_1,\ldots,x_n)$ is non-trivially PR
if and only if there exist infinite hypernatural numbers
$\xi_1\ueq\ldots\ueq\xi_n$ with
${}^*\!f(\xi_1,\ldots,\xi_n)=0$.
\end{cor}

\begin{proof}
By the previous proposition, we can
pick elements $\xi_1\ueq\ldots\ueq\xi_n$ 
with ${}^*f(\xi_1,\ldots,\xi_n)=0$
and $|\{\xi_1,\ldots,\xi_n\}|>1$. If 
one of the $\xi_i$ equals a finite $k\in\N$, then
by property (U1) we would have $\xi_j=k$ for all $j=1,\ldots,n$,
a contradiction.
\end{proof}

As a first easy example of application of $u$-equivalence, let us
prove the following fact.\footnote
{~This basic property was first pointed out by H. Lefmann \cite{L91}
for bijective functions $f$, in the context of rings.
Indeed, Proposition \ref{compose}
also holds if one replaces $\N$ with an arbitrary
ring $R$ (of course in this case PR means PR on $R$).}

\begin{prop}\label{compose}
Let $f:\N^n\to\N$ and let $\varphi:\N\to\N$.
If $f(\varphi(x_1),\ldots,\varphi(x_n))$ is PR then
$f(x_1,\ldots,x_n)$ is PR.
Moreover, if $\varphi$ is onto, then one has
the equivalence: $f(x_1,\ldots,x_n)$ is PR $\Leftrightarrow$
$f(\varphi(x_1),\ldots,\varphi(x_n))$ is PR.
\end{prop}

\begin{proof}
Pick $\xi_1\ueq\ldots\ueq\xi_n$ such that
$\hf({}^*\varphi(\xi_1),\ldots,{}^*\varphi(\xi_n))=0$, and let $\eta_i={}^*\varphi(\xi)_i$.
Then $\eta_1\ueq\ldots\ueq\eta_n$ 
and trivially $\hf(\eta_1,\ldots,\eta_n)=0$.
If $\varphi$ is onto, pick $\psi:\N\to\N$ such that
$\varphi\circ\psi$ is the identity, and consider $\eta_i={}^*\psi(\xi_i)$.
Then $\eta_1\ueq\ldots\ueq\eta_n$ are such that
$\hf({}^*\varphi(\eta_1),\ldots,{}^*\varphi(\eta_n))=\hf(\xi_1,\ldots,\xi_n)=0.$
\end{proof}

When dealing with polynomials in several variables,
it is convenient to use the multi-index notation.
Let us fix the terminology.

\begin{itemize}
\item
An $n$-dimensional \emph{multi-index}
is an $n$-tuple $\alpha=(\alpha_1,\ldots,\alpha_n)\in\N_0^n$\,;
\item
$\alpha\le\beta$ means that $\alpha_i\le\beta_i$ for all $i=1,\ldots,n$\,;
\item
$\alpha<\beta$ means that $\alpha\le\beta$ and $\alpha\ne\beta$\,;
\item
If $\x=(x_1,\ldots,x_n)$ is vector
and $\alpha=(\alpha_1,\ldots,\alpha_n)$ is a multi-index, the product
$\prod_{i=1}^n x_i^{\alpha_i}$
is denoted by $\x^\alpha$\,;
\item
The \emph{length} of a multi-index $\alpha=(\alpha_1,\ldots,\alpha_n)$ 
is $|\alpha|=\sum_{i=1}^n\alpha_i$\,;
\item
A set $I$ of $n$-dimensional multi-indexes having all
the same length is called
\emph{homogeneous}\,;
\item
Polynomials $P\in\Z[x_1,\ldots,x_n]$ are written
in the form $P(\x)=\sum_\alpha c_\alpha\x^\alpha$
where $\alpha$ are multi-indexes\,;
\item
The \emph{support} of $P$ is the finite set
$\text{supp}(P)=\{\alpha\mid c_\alpha\ne 0\}$\,;
\item
A polynomial $P(\x)=\sum_\alpha c_\alpha\x^\alpha$
is \emph{homogeneous} if $\text{supp}(P)$ is a
homogeneous set of indexes.
\end{itemize}

\begin{defn}
{\rm Let $P(\x)=\sum_\alpha c_\alpha\x^\alpha\in\Z[x_1,\ldots,x_n]$.
We say that a multi-index $\alpha\in\text{supp}(P)$ is \emph{minimal} 
if there are no $\beta\in\text{supp}(P)$ with $\beta<\alpha$.
The notion of \emph{maximal} multi-index is defined similarly.

A nonempty set $J\subseteq\text{supp}(P)$
is called a \emph{Rado set of indexes}
if for every $\alpha,\beta\in J$
there exists a nonempty $\Lambda\subseteq\{1,\ldots,n\}$
with $\sum_{i\in\Lambda}\alpha_i=\sum_{i\in\Lambda}\beta_i$.}
\end{defn}

Notice that every singleton $\{\alpha\}\subseteq\text{supp}(P)$
is trivially a Rado set. 
When $P(x_1,\ldots,x_n)=c_1x_1+\ldots+c_n x_n$
is a linear polynomial with no constant term, then
we can write $P=\sum_{s=1}^n c_s\x^{\alpha(s)}$
where $\alpha(s)$ is the multi-index where 
the $s$-th entry is $1$, and all other entries are $0$. In this
case, every nonempty
$J\subseteq\text{Supp}(P)=\{\alpha(1),\ldots,\alpha(n)\}$
is a Rado set of both minimal and maximal indexes.

\begin{thm}\label{general necessary condition for PR}
Let $P(\x)=\sum_\alpha c_\alpha\x^\alpha\in\Z[x_1,\ldots,x_n]$
be a polynomial with no constant term.
Suppose there exists a prime $p$ such that:
\begin{enumerate}
\item
$\sum_\alpha c_\alpha z^{|\alpha|}\equiv 0\mod p$ 
has no solutions $z\not\equiv 0$\,;
\item
For every Rado set $J$ of minimal indexes,
$\sum_{\alpha\in J} c_\alpha z^{|\alpha|}\equiv 0\mod p$ 
has no solutions $z\not\equiv 0$.
\end{enumerate}
Then $P(\x)$ 
is not PR, except possibly for constant solutions $x_1=\ldots=x_n$.
\end{thm}

\begin{proof}
By contradiction, let us suppose that the polynomial
$P(\x)$ is non-trivially PR, and pick infinite
$\xi_1\ueq\ldots\ueq\xi_n$ such that 
$$P(\xxi)\ =\ \sum_\alpha c_\alpha\xxi^\alpha=0.$$ 

Pick a prime $p$ as given by the hypothesis,
and write the numbers $\xi_i$ in the following form:
$$\xi_i\ =\ a_i+\zeta_i\, p^{\tau_i}$$
where $0\le a_i\le p-1$, where $\zeta_i$ is not divisible by $p$,
and where $\tau_i\ge 1$. Denote by $b_i\in\{1,\ldots,p-1\}$
the number such that $\zeta_i\equiv b_i\mod p$.

Let $f:\N\to\{0,1,\ldots,p-1\}$ be the function where 
$f(m)\equiv m\mod p$; let $g:\N\to\N$ be the function
where $g(m)$ is the greatest exponent of $p$ that divides $m-f(m)$;
and let $h:\N\to\{1,\ldots,p-1\}$ be the function where 
$h(m)\equiv(m-f(m))/p^{g(m)}\mod p$.
Notice that ${}^*f(\xi_i)=a_i$, ${}^*g(\xi_i)=\tau_i$
and ${}^*h(\xi_i)=b_i$. So, 
the $u$-equivalences $\xi_1\ueq\ldots\ueq\xi_n$ imply that
$a_1\ueq\ldots\ueq a_n$, $\tau_1\ueq\ldots\ueq\tau_n$,
and $\zeta_1\ueq\ldots\ueq \zeta_n$.
Since finite $u$-equivalent numbers are necessarily equal,
there exist $0\le a\le p-1$ and $1\le b\le p-1$ such that 
$a_i=a$ and $b_i=b$ for all $i$. Now,
$$0\ =\ P(\xxi)\ \equiv\ 
\sum_{\alpha}c_\alpha a^{|\alpha|}\mod p,$$
and hence, by the hypothesis (1), it must be $a=0$.
In consequence,
$$\xxi^\alpha\ =\  
\zzeta^\alpha\cdot p^{\,\sum_{i=1}^n\!\alpha_i\tau_i}$$
where $\zzeta^\alpha\equiv b^{|\alpha|}\not\equiv 0\mod p$.
Now let 
$\sigma=\min\left\{\sum_{i=1}^n\alpha_i\tau_i\mid 
\alpha\in\text{supp}(P)\right\}$, and let
$J=\{\alpha\mid \sum_{i=1}^n\alpha_i\tau_i=\sigma\}$.
We have that
$$0\ =\ \sum_\alpha c_\alpha\xxi^\alpha\ =\ 
p^\sigma\cdot\left(\sum_{\alpha\in J}c_\alpha\zzeta^\alpha+
\sum_{\beta\notin J}
c_\beta\, \zzeta^\beta\,p^{\,(\sum_{i=1}^n\!\beta_i\tau_i)-\sigma}\right).$$
Then
$$\sum_{\alpha}c_\alpha\zzeta^\alpha\ \equiv\ 
\sum_{\alpha\in J}c_\alpha b^{|\alpha|}\ \equiv\ 0\mod p.$$
This shows that the equation
$$\sum_{\alpha\in J}c_\alpha z^{|\alpha|}\ \equiv\ 0 \mod p$$
has the solution $b\not\equiv 0\mod p$.
We will reach a contradiction with hypothesis (2),
by showing that $J$ is a Rado set of minimal indexes.
Notice first that $J$ only contains minimal indexes; indeed,
if $\beta<\alpha\in J$ then $\sigma-\sum_{i=1}^n\beta_i\tau_i=
\sum_{i=1}^n(\alpha_i-\beta_i)\tau_i>0$
since all $\tau_i\ge 1$, and so $\beta\notin\text{supp}(P)$.
Let us now prove that $J$ is a Rado set. 
Take any two distinct indexes $\alpha,\beta\in J$.
(If $J$ is a singleton, the thesis is trivial.)
Then $\sum_{i=1}^n\beta_i\tau_i-
\sum_{i=1}^n\alpha_i\tau_i=\sigma-\sigma=0$.
Since $\tau_1\ueq\ldots\ueq\tau_n$,
by the nonstandard characterization,
the equation $\sum_{i=1}^n(\beta_i-\alpha_i)y_i=0$
is PR. In consequence, by Rado's theorem, there
exists a nonempty $\Lambda\subseteq\{1,\ldots,n\}$ such that 
$\sum_{i\in\Lambda}(\beta_i-\alpha_i)=0$, as desired.
\end{proof}

The range of Diophantine equations covered by 
Theorem \ref{general necessary condition for PR}
is quite large. Two easy examples are the following.

\begin{ex}
{\rm
Let $P(x_{1},x_{2},x_{3})=x_{1}^{2}x_{2}-2x_{3}$. 
Pick any prime number $p$ with $p\equiv 3$ 
or $p\equiv 5\mod 8$, so that $2$ is not a quadratic residue modulo $p$. 
Then condition (1) of Theorem \ref{general necessary condition for PR} 
is satisfied because $z^{3}-2z\equiv 0$ iff $z\equiv 0$,
and also condition (2) is easily verified.
Since it has no constant solutions
$x_1=x_2=x_3$, we can conclude that
$P(x_{1},x_{2},x_{3})$ is not PR.}
\end{ex}

\begin{ex}
{\rm Let $P(x_{1},\dots,x_{p+1},y,z)=\prod_{i=1}^{p+1}x_{i}-yz+z$, where $p$ is a prime number. Conditions (1) and (2) of Theorem \ref{general necessary condition for PR} are immediate consequences of Fermat's Little Theorem, hence $P(x_{1},\dots,x_{p+1},y,z)$ is not PR.}
\end{ex}

As a particular case of Theorem \ref{general necessary condition for PR},
we obtain a result about homogeneous equations, 
first proved by the second named author in \cite{LB12}.

\begin{cor}
Let $P(\x)=\sum_\alpha c_\alpha\x^\alpha\in\Z[x_1,\ldots,x_n]$ 
be an homogeneous polynomial.
If for every nonempty $\Gamma\subseteq \text{supp}(P)$
one has $\sum_{\alpha\in\Gamma} c_{\alpha}\neq 0$, 
then $P(\x)$ is not PR.
\end{cor}

\begin{proof}
If $d$ is the degree of $P(\x)$, then
for every prime number $p>\sum_{\alpha}|c_{\alpha}|$,
we have that 
$\sum_{\alpha\in\Gamma}c_{\alpha}z^{|\alpha|}=
z^{d}\cdot \sum_{\alpha\in\Gamma}c_{\alpha}\equiv 0 \mod p$ 
if and only if $z\equiv 0\mod p$, and so condition (1)
of Theorem \ref{general necessary condition for PR} is satisfied.
Notice that the hypothesis directly implies that also
condition (2) holds, and so
we can conclude that $P(\x)$ is not PR. 
\end{proof}

While the above corollary provides
a necessary condition for homogeneous Diophantine
equations to be PR, let us mention
that H. Leifmann \cite[Fact 2.8]{L91} isolated
a sufficient condition for a special class of homogeneous
quadratic equations to be PR.\footnote
{~Precisely, $\sum_{i=1}^n c_i x_i^2$ is PR
if there exists a nonempty $\Lambda\subseteq\{1,\ldots,n\}$
and there exist numbers $a\in\N$ and $b_i\in\Z$ such that:
(1) $\sum_{i\in\Lambda}c_i=0$;
(2) $\sum_{i\in\Lambda}b_ic_i=0$;
(3) $\sum_{i\in\Lambda}b_i^2c_i+a^2\sum_{i\notin\Lambda}c_i=0$.}

Another necessary condition for PR 
applies when every monomial of $P(\x)$ contains
a single variable, \emph{i.e.} 
when $P$ has the form $P_1(x_1)+\ldots+P_n(x_n)$.

\begin{thm}\label{general non-PR result}
Let $P(\x)=\sum_\alpha c_\alpha\x^\alpha\in\Z[x_1,\ldots,x_n]$
be a polynomial with no constant term and where every
monomial contains a single variable. 
If $P(\x)$ is non-trivially PR 
then $\sum_{\alpha\in I}c_\alpha=0$
for some nonempty maximal homogeneous set of
indexes $I\subseteq\text{supp}(P)$.
\end{thm}

\begin{proof}
By the hypothesis, for every $\alpha\in\text{supp}(P)$
there exists a unique $i$ such that $\alpha_i\ne 0$.
If we take such an $i$ and let $s=\alpha_i=|\alpha|$,
then we can write $c_{i,s}x_i^s$ 
in place of $c_\alpha\x^\alpha$ with no ambiguity.

Let $d$ be the greatest degree
and $k$ the least degree of a monomial of $P$.
If $\Gamma(s)=\{i\mid \exists\alpha\ \alpha_i=s\}$ and
$\Lambda(i)=\{\alpha_i\mid \alpha_i\ne 0\}$ then
$$P(\x)\ =\ \sum_{s=k}^d\sum_{|\alpha|=s}c_\alpha\x^\alpha\ =\ 
\sum_{s=k}^d\sum_{i\in\Gamma(s)}c_{i,s}\,x_i^s\ =\ 
%\sum_{i=1}^n P_i(x_i)\ \ \text{where} P_i(x_i)=
\sum_{i=1}^n \sum_{s\in\Lambda(i)}c_{i,s}x_i^s.$$

By the nonstandard characterization of non-trivial PR,
we can pick infinite $\xi_1\ueq\ldots\ueq\xi_n$ such that 
$P(\xxi)=0$. Now fix any finite number $p\ge 2$,
and write the numbers $\xi_i$ in base $p$:
$$\xi_i\ =\ \sum_{t=0}^{\tau_i}a_{i,t}\,p^{\tau_i-t}$$
where $0\le a_{i,t}\le p-1$ and $a_{i,0}\ne 0$.
In particular, $p^{\tau_i}\le \xi_i<p^{\tau_i+1}$.

Let $s_*\tau_*=\max\{s\,\tau_i\mid k\le s\le d,\, i\in\Gamma(s)\}$,
let $I_*=\{i\in\Gamma(s_*)\mid \tau_i=\tau_*\}$,
and decompose $P(\xxi)=\Theta+\Psi+\Phi$, where:
\begin{itemize}
\item
$\Theta=\sum_{i\in I_*}c_{i,s_*}\,\xi_i^{s_*}$\,;
\item
$\Psi=\sum_{i\in\Gamma(s_*)\setminus I_*}c_{i,s_*}\,\xi_i^{s_*}$\,;
\item
$\Phi=\sum_{s\ne s_*}\sum_{i\in\Gamma(s)}c_{i,s}\,\xi_i^{s}$.
\end{itemize}

In the sequel, for numbers $\xi,\xi'\in\hN$,
we will write $\xi\lll\xi'$ to mean that $\xi'-\xi$ is infinite.

\smallskip
\begin{lem}\label{Tecnico}
\

\begin{enumerate}
\item
$\Theta=\left(\sum_{i\in I_*}c_{i,s_*}\right)\zeta+\Theta'$
where $\zeta\ge p^{s_*\tau_*}$ and $|\Theta'|\lll p^{s_*\tau_*}$.
\item
$|\Psi|\lll p^{s_*\tau_*}$.
\item
$|\Phi|\lll p^{s_*\tau_*}$.
\end{enumerate}
\end{lem}

Since $P(\xxi)=\Theta+\Psi+\Phi=0$, the above inequalities
imply that the sum of coefficients $\sum_{i\in I_*}c_{i,s_*}=0$.
We reach the thesis by noticing that 
$I=\{\alpha\in\text{Supp}(P)\mid
\exists i\  \tau_i=\tau_* \ \&\ \alpha_i=s_*\}$ is 
a nonempty homogeneous set of maximal indexes,
and that $\sum_{\alpha\in I}c_\alpha=\sum_{i\in I_*}c_{i,s_*}$.

\smallskip
We are left to prove the Lemma; let us start with some preparatory work.

Let $\varphi:\N\to\N_0$ be the function where 
$p^{\varphi(m)}\le m<p^{\varphi(m)+1}$;
and for every $t\in\N_0$, let $\psi_t(m):\N\to\{0,1,\ldots,p-1\}$
be the function where $\psi_t(m)$ is the $(t+1)$-th digit from the left
when $m$ is written in base $p$. 
Then ${}^*\varphi(\xi_i)=\tau_i$
and  ${}^*\psi_t(\xi_i)=a_{i,t}$, and the 
$u$-equivalences $\xi_1\ueq\ldots\ueq\xi_n$ imply that
$\tau_1\ueq\ldots\ueq\tau_n$ and $a_{1,t}\ueq\ldots\ueq a_{n,t}$.
Since finite $u$-equivalent numbers are necessarily equal,
it is $a_{1,t}=\ldots=a_{n,t}$. 
Then, by \emph{overspill}, there exists an infinite $\nu\in\hN$
and numbers $b_t\in\{0,\ldots,p-1\}$
for $t\le\nu$ such that
$a_{i,t}=b_t$ for every $i=1,\ldots,n$ and
for every $t\le\nu$.
Let us denote by 
$$\zeta_i\ =\ 
\sum_{t=0}^\nu b_t\, p^{\tau_i-t}.$$

We will use the following decomposition:

\begin{itemize}
\item
\emph{For every $a\in\N$ one has
$\xi_i^a=\zeta_i^a+\vartheta_{i,a}$
where 
$p^{a\tau_i}\le\zeta_i^a\le\xi_i^a<p^{a\tau_i+a}$ and
$\vartheta_{i,a}\lll p^{a\tau_i}$.}
\end{itemize}

Since $p^{\tau_i}\le\zeta_i\le\xi_i<p^{\tau_i+1}$,
it directly follows that
$p^{a\tau_i}\le \zeta_i^a\le\xi_i^a<p^{a\tau_i+a}$;
besides, the difference
$\eta_i=\xi_i-\zeta_i=
\sum_{t=\nu+1}^{\tau_i}a_{i,t} p^{\tau_i-t}<p^{\tau_i-\nu}$.
Now, $\xi_i^a=(\zeta_i+\eta_i)^a=\zeta_i^a+\vartheta_{i,a}$ where
$\vartheta_{i,a}=\sum_{j=1}^a\binom{a}{j}\zeta_i^{a-j}\eta_i^j$.
Pick a large enough 
$\ell_i\in\N$ so that $\binom{a}{j}<p^{\ell_i}$ for all $j$.
Then
\begin{multline*}
\vartheta_{i,a}\ <\ 
p^{\ell_i}\sum_{j=1}^a\zeta_i^{a-j}\eta_i^j\ <\ 
p^{\ell_i}\sum_{j=1}^{a}
\left(p^{\tau_i+1}\right)^{a-j}\cdot\left(p^{\tau_i-\nu}\right)^j\ =
\\
=\ p^{\ell_i}\sum_{j=1}^a p^{a\tau_i+a-j(\nu+1)}\ <\ 
p^{2\ell_i}\, p^{a\tau_i+a-\nu-1}\ =\ 
p^{a\tau_i-(\nu-2\ell_i-a+1)}\ \lll\ p^{a\tau_i}.
\end{multline*}
Indeed, since $\nu$ is infinite, also 
$\nu-2\ell_i-a+1$ is infinite.

We can now prove points (1), (2) and (3) of Lemma \ref{Tecnico}.

\smallskip
1. With the notation introduced above,
$$\Theta\ =\ \sum_{i\in I_*}c_{i,s_*}\zeta_i^{s_*}+
\sum_{i\in I_*}c_{i,s_*}\vartheta_{i,s_*}.$$

For every $i\in I_*$, by the above estimates we know that
$\zeta_i=\sum_{t=0}^\nu b_t p^{\tau_*-t}=\zeta_*\ge p^{\tau_*}$
and $\vartheta_{i,s_*}\lll p^{s_*\tau_i}=p^{s_*\tau_*}$.
Then $\sum_{i\in I_*}c_{i,s_*}\zeta_i^{s_*}=
\left(\sum_{i\in I_*}c_{i,s_*}\right)\zeta$ where 
$\zeta=\zeta_*^{s_*}\ge p^{s_*\tau_*}$, and
$|\Theta'|=|\sum_{i\in I_*}c_{i,s_*}\vartheta_{i,s_*}|\le
\sum_{i\in I_*}|c_{i,s_*}|\,\vartheta_{i,s_*}\lll p^{s_*\tau_*}$.

\smallskip
2. If $i\in\Gamma(s_*)\setminus I_*$ then $s_*\tau_i<s_*\tau_*$,
and since $s_*\tau_i\ueq s_*\tau_*$, it follows that $s_*\tau_i\lll s_*\tau_*$.
Then
$$|\Psi|\ \le\ \sum_{i\in\Gamma(S_*)\setminus I_*}|c_{i,s_*}|\,\xi_i^{s_*}\ \le\ 
\sum_{i\in\Gamma(S_*)\setminus I_*}|c_{i,s_*}|\,p^{s_*\tau_i}\ \lll\ 
p^{s_*\tau_*}.$$

\smallskip
3. Let us show that for every $s\ne s_*$ and
for every $i\in\Gamma(s)$, one has $s\tau_i\lll s_*\tau_*$.
By the definition of $s_*\tau_*$, clearly $s\tau_i\le s_*\tau_*$.
If by contradiction $s_*\tau_*-s\tau_i=h\in\Z$,
then we would have ${}^*f(\tau_i)=\tau_*$ where $f:\N\to\N$ is the
function $f(m)=\lfloor(sm+h)/s_*\rfloor$.\footnote
{~ By $\lfloor\,x\,\rfloor$ we denote the integer part of $x$.}
Since $\tau_*\ueq\tau_i$, it would follow
that ${}^*f(\tau_i)=\tau_i$, and hence $(s_*-s)\tau_i=h$.
But $\tau_i$ is infinite while $h\in\Z$,
and so we must conclude that $s_*=s$, against our hypothesis.
The thesis is directly obtained by the following inequalities:
$$|\Phi|\ \le\ \sum_{s\ne s_*}\sum_{i\in\Gamma(s)}|c_{i,s}|\,\xi_i^{s}\ \le\ 
\sum_{s\ne s_*}\sum_{i\in\Gamma(s)}|c_{i,s}|\,p^{s\tau_i}\ \lll\
p^{s_*\tau_*}.\qedhere$$
\end{proof}

As a straight consequence of the Theorem \ref{general non-PR result} we obtain the following generalization of (one implication in) Rado's Theorem.

\begin{cor}\label{carino}
If the Diophantine equation
$$\sum_{i=1}^{n} c_{i}x_{i}^{d_{i}}=0$$
is PR then the following
``Rado's condition'' is satisfied:
\begin{itemize}
\item
``There exists a nonempty set $J\subseteq\{1,\dots,n\}$ such that $\forall i,j\in J \ d_{i}=d_{j}$ and
$\sum_{j\in J} c_{j}=0$.''
\end{itemize}
\end{cor}

Another simple consequence of Theorem \ref{general non-PR result} is the following.

\begin{cor}
Let us consider a Diophantine equation of the form
$$\sum_{i=1}^{n} c_{i}x_{i}^{k}\ =\ P(y)$$
where $P(y)$ is a polynomial with no constant term
of degree $d\ne k$. If
for every nonempty set $\Gamma\subseteq\{1,\dots,n\}$ 
one has $\sum_{i\in\Gamma}c_{i}\neq 0$, then the above
equation is not PR.
\end{cor}

Finally, by combining Theorems \ref{Generalized Rado} and
\ref{general non-PR result} one obtains an
extension of Rado's Theorem to a large family of 
nonlinear polynomials.

\begin{cor}\label{characterization}
Let $n\ge 3$. A polynomial of the form
$$Q(x_{1},\dots,x_{n},y):=c_{1}x_{1}+\dots+c_{n}x_{n}+P(y)$$
where $P$ is nonlinear is non-trivially PR if and only if
there exists a nonempty subset $J\subseteq\{1,\dots,n\}$ such that
$\sum_{j\in J} c_{j}=0$.
\end{cor}

\begin{proof} 
One implication is Theorem \ref{Generalized Rado}.
Conversely, by Theorem \ref{general non-PR result}, 
there exists a nonempty maximal homogeneous set of indexes 
$I\subseteq\text{supp}(Q)$ such that 
$\sum_{\alpha\in I} c_{\alpha}=0$. 
As $P(y)$ is nonlinear, all the indexes in $I$ correspond to monomials 
of $a_{1}x_{1}+\dots+a_{n}x_{n}$, and hence 
$\sum_{\alpha\in I} c_{\alpha}=0$ actually
means $\sum_{j\in J}c_j=0$ for a suitable nonempty 
$J\subseteq\{1,\ldots,n\}$.
\end{proof}

Let us itemize some explicit examples of polynomials whose non-PR 
is proved by our results.

\begin{ex}
{\rm 
The equation
$x-2y=P(z)$ is \emph{not} partition regular
for any nonlinear polynomial $P(z)\in\Z[z]$. 
This gives a negative answer to Question 11 (iii)
posed by V. Bergelson in \cite{B96}.}
\end{ex}

\begin{ex}
{\rm The equation $x+y=z^2$ is not PR, except
for the constant solution $x=y=z=2$. (This was
first proved by P. Csikv\'ari, K. Gyarmati and A. S\'ark\"{o}zy
in \cite{CGS12}.)}
\end{ex}

\begin{ex}
{\rm A. Khalfah and E. Szemer\'edi \cite{KS06} proved that 
if $P(z)\in\Z[z]$ takes even values
on some integer, then for every finite coloring
the equation $x+y=P(z)$ has a solution where $x$ and $y$ are monochromatic.
However, as a consequence of Corollary \ref{characterization}, 
it is never the case
that $x+y=P(z)$ is partition regular when $P$ is nonlinear.}
\end{ex}

\begin{ex}
{\rm In \cite{DNR16},
it is proved that the following polynomials $x^{n}+y^{m}=z^{k}$ 
are not PR for $k\notin\{n,m\}$. This result is
obtained as a particular case of Corollary \ref{carino}. }
\end{ex}

\section{Final remarks and open questions}

In the last years, the interest on problems related to 
the partition regularity of nonlinear Diophantine equations has 
been rising constantly (see, \emph{e.g.},
\cite{KS06,H11,CGS12,SV14,LB14,LB15,BJM15,DNR16,FH16,HKM16}).  
We hope that this paper will contribute to a general Ramsey 
theory of nonlinear Diophantine equations. In this direction, 
we think that at least four distinct directions of research are 
worth pursuing.

The first one is trying to extend our results so to fully characterize 
the class of nonlinear PR Diophantine equations on $\N$
in ``Rado's style'', \emph{i.e.} by means of decidable simple 
conditions on coefficients and exponents.\footnote
{~Here the word ``decidable'' has the precise sense
as defined in compatibility theory to formalize
the idea of an ``effective method'' .}
As the general problem seems highly complicated,
it would surely be helpful to start by isolating other classes 
of PR and non-PR equations. 
For example, we think that it would be really interesting to 
find a solution to the following.

\bigskip
\noindent
\textbf{Open Problem 1.}
\emph{Under the additional assumption that the given 
equation admits solutions in $\N$,
can the implication in Theorem \ref{general non-PR result} or, at least, 
in Corollary \ref{carino}, be reversed? \footnote
{~The hypothesis on the existence of solutions is
needed, as otherwise the conjecture would be false, 
as shown, \emph{e.g.}, by 
Fermat equations $x^n+y^n=z^n$ with $n\ge 3$.}}

\bigskip
Notice that a positive answer to this question 
would entail the PR of the Pythagorean equation 
$x^2+y^2=z^2$, which is probably the most investigated
open problem in this field. It is our opinion that nonstandard 
analysis could play an important role in this research,
also in the positive direction of PR results.
Indeed, techniques based on $u$-equivalence have already 
been used by the second named author in \cite{LB14}
to prove the PR of several classes of nonlinear equations.

A second possible direction of research is to study 
the PR of nonlinear Diophantine equations on sets of numbers
different from the natural numbers.
In this respect, let us point out a few facts.

\begin{enumerate}
\item
A \emph{homogeneous} Diophantine equation is PR on 
$\N$ if and only if it is 
PR on $\Z$ if and only if it is PR on $\Q$.

\item
There are homogeneous Diophantine equations
that are PR on the positive reals $\R_{>0}$ but not on $\N$.

\item
For non-\emph{homogeneous} equations,
the equivalences in (1) do not hold.
\end{enumerate}

The ``only if'' implications in (1) are trivial.
Conversely, let us observe that if $P(x_{1},\dots,x_{n})=0$ 
is a homogeneous Diophantine equation that is PR on $\Q$, then every 
ultrafilter $\U\in\overline{K(\beta\Q,\odot)}$ is a witness.
(This follows from the analogues of Theorem \ref{alpha} and 
Corollary \ref{ultrahomogeneous} for $\Q$.) 
Since the set $\N$ is thick in the group
$(\Q,\cdot)$, and hence piecewise syndetic,
we can pick $\U\in\overline{K(\beta\Q,\odot)}$ such that $\N\in\U$.\footnote
{~Recall that a subset $A\subseteq S$ of a semigroup $(S,\cdot)$
is piecewise syndetic if and only if it belongs to
some ultrafilter in $\overline{K(\beta S,\cdot)}$
(see Corollary 4.41 in \cite{HS11}.)}
Then $\U_\N=\{B\cap\N\mid B\in\U\}$ is an ultrafilter on
$\N$ that witnesses the PR on $\N$ of the equation $P(x_{1},\dots,x_{n})=0$.

Easy examples to show (2) are given by all Fermat equations 
$x^n+y^n=z^n$ with $n\ge 3$, which do not admit solutions in $\N$
but are PR on $\R_{>0}$. Indeed, Schur equation $x+y=z$ is PR,
and by taking the function $\varphi(x)=x^{n}$, which is onto $\R_{>0}$,
one can apply the analogue of Proposition \ref{compose} to the semigroup $\R_{>0}$.

As for (3), consider, \emph{e.g.}, the equation $x_{1}y_{1}-x_{2}=0$. 
By the multiplicative Rado's Theorem, that equation is PR on $\N$, 
and hence it is PR on $\Z$. 
Then, by Proposition \ref{compose} for the group $\Z$
applied to the function $f(x)=-x$, we obtain that also 
$x_{1}y_{1}+x_{2}=0$ is PR on $\Z$, whilst it is has no
solutions in $\N$. 

A general question that arises naturally is the following.

\bigskip
\noindent
\textbf{Open Problem 2.}
\emph{Are there simple decidable conditions
under which a given (non-homogeneous) Diophantine equation 
with no constant term
is PR on $\N$ if and only if it is PR on $\Z$ 
if and only if it is PR on $\Q$?}

\bigskip
A problem that seems to have its own peculiarities
is about the PR of Diophantine equations on finite fields.
About this, a relevant result has been recently
obtained by P. Csikv\'ari, K. Gyarmati and A. S\'ark\"{o}zy \cite{CGS12},  
who proved the PR of every Fermat equation $x^n+y^n=z^n$
on sufficiently large finite fields $\mathbb{F}_p$
(with $xyz\ne 0$).

It seems natural to ask whether the techniques
used in this paper may help towards the following.

\bigskip
\noindent
\textbf{Open Problem 3.}
\emph{Are there simple ``Rado-like'' necessary
and sufficient conditions
under which a given Diophantine equation 
with no constant term is PR on sufficiently large
finite fields $\mathbb{F}_p$?}

\bigskip
Finally, another really wide direction of research is investigating
the PR of finite and infinite systems of nonlinear Diophantine equations.
Whilst certain particular results are already known, such as the multiplicative 
version of Hindman's Theorem, general results in this area are still missing.
It is worth remarking that although extensively studied in the recent
literature (see, \emph{e.g.},
\cite{HL93,DHLL95,HLS02,HLS03,HL06,H07,LR07,DHLS14,
GHL14,SV14,HLS15,BHLS15,BJM15}), 
also infinite linear systems are not fully understood yet.
In order to adapt our nonstandard techniques to 
infinite systems, one would need a  
characterization of PR systems in terms of
$u$-equivalence. The characterization given in 
Corollary \ref{U-equivalence Characterization} is easily
generalized to finite systems, but we do not see how
to extend it to infinite systems. 

\bigskip
\noindent
\textbf{Open Problem 4}
\emph{Is there a characterization of PR infinite systems of 
Diophantine equations in terms of $u$-equivalence?
(Or, equivalently, by means of ultrafilters?)}

\end{document}